\RequirePackage{fix-cm}
\documentclass[smallextended]{svjour3}       
\smartqed  
\usepackage{graphicx}
\usepackage[all]{xy}
\usepackage[centertags]{amsmath}
\usepackage{latexsym}
\usepackage{amsfonts}
\usepackage{amssymb}
\usepackage{url}
\usepackage{color}

\newcommand{\N}{\mathbb{N}}

\newcommand{\Z}{\mathbb{Z}}

\newcommand{\qbinomial}[3]{\mbox{$
\left[ \!
\begin{array}{c}
#1\\
#2
\end{array}\!\right]_{
{#3}} $} }

\def \A {\mathcal{A}}

\def \Dq {{D_q}}

\newcommand{\Bern}[1]{\mathbf{B}_{#1,q}}
\newcommand{\Euln}[1]{\mathbf{E}_{#1,q}}

\newcommand{\Bn}[2]{{\mathbf{B}}^{(#1)}_{#2,q}}
\newcommand{\En}[2]{{\mathbf{E}}^{(#1)}_{#2,q}}

\begin{document}

\title{On a new $q$-analogue of Appell polynomials 
}

\titlerunning{On a new $q$-analogue of Appell polynomials }        

\author{P. Njionou Sadjang }


\institute{P. Njionou Sadjang \at
              Faculty of Industrial Engineering, University of Douala \\
              Tel.: +237 677 034 787\\
              \email{pnjionou@yahoo.fr}           
%
}

\date{Received: date / Accepted: date}

\maketitle

\begin{abstract}
A new $q$-analogue of Appell polynomial sequences and their generalizations are introduced and their main characterizations are proved. As consequences new $q$-analogue of Bernoulli and Euler polynomials and numbers is introduced, their main representations are given. 
\keywords{ Appell polynomial set, $q$-Bernoulli polynomials, $q$-Euler polynomials, orthogonal polynomials, quasi-orthogonal polynomials, $q$-difference equation.    }
 \subclass{33C65\and 33C45\and 33D05\and 33D45\and 11B68. }
\end{abstract}

\section{Introduction}\label{intro}
\noindent   Throughout this paper, we use the following standard notations 
\[\N:=\{1,2,3,\ldots\},\quad \N_{0}=\{0,1,2,3,\ldots\}=\N\cup\{0\}.\]

\noindent Let $P_n(x)$, $n=0,\, 1,\,2,\,\ldots$ be a polynomial set, i.e. a sequence of polynomials with $P_n(x)$ of exact degre $n$. Assume further that 
\[\dfrac{dP_n(x)}{dx}=P'_n(x)=nP_{n-1}(x) \quad \textrm{for}\quad n=0,\, 1,\,2,\,\ldots.\] Such polynomial sets are called Appell sets and received considerable attention since P. Appell \cite{appell1880} introduced them in 1880. \\ 

\noindent Let $q$  be an arbitrary real or complex number and define the $q$-derivative \cite{kac} of a function $f(x)$ by means of
\begin{equation}
 \Dq f(x)=\dfrac{f(x)-f(qx)}{(1-q)x},
\end{equation}
which furnishes a generalization of the differential operator $\dfrac{d}{dx}$.

\noindent A basic ($q$-)analogue of Appell sequences was first introduced by Sharma and Chak \cite{sharma} and they called them $q$-harmonic. Later, Al-Salam \cite{al-salam} studied these families and referred to it as $q$-Appell sets in analogy with ordinary Appell sets. Note that both Sharma and Al-Salam defined the so-called $q$-Appell sets as those sets $\{P_n(x)\}_{n=0}^{\infty}$ which satisfy 
\begin{equation}\label{eq1}
 D_qP_n(x)=[n]_{q}P_{n-1}(x),\quad n=0,\; 1,\; 2,\;3,\;\cdots
\end{equation}
where $[n]_q=(1-q^n)/(1-q)$. Note that when $q\to 1$, \eqref{eq1} reduces to 
\[\dfrac{dP_n(x)}{dx}=nP_{n-1}(x), \] so that we may think of $q$-Appell sets as a generalization of Appell sets. 
We will call these polynomial sets {\it $q$-Appell sets of type I}.

\noindent The purpose of this paper is to study the class of polynomial sets $\{P_n(x)\}$ which satisfy 
\begin{equation}\label{eq2}
 D_qP_n(x)=[n]_{q}P_{n-1}(qx),\quad n=0,\; 1,\; 2,\;3,\;\cdots
\end{equation}
Again \eqref{eq2} reduces to $\dfrac{d}{dx} P_n(x)=nP_{n-1}(x)$ as $q\to 1$ so that we may also think of these sets as another $q$-generalization of Appell sets. 
We will call these polynomial sets {\it $q$-Appell sets of type II}. 


When a possible confusion on the $q$-derivative will happen, we will use the notation $D_q^{\{t\}}$ to show that the derivative is done with respect to the variable $t$, otherwise, the derivative will be done with respect to the variable $x$.


This paper is organized as follows. In section 2  we give some preliminary definitions that are useful for the sequel. In section 3 we provide several characterizations of $q$-Appell polynomial sequences of type II, introduce new $q$-Bernoulli and $q$-Euler polynomials as examples of $q$-Appell polynomial sequences of type II. Some of the representations of the new polynomial sequences are also given. In section 4, we give some algebraic structures of the set of $q$-Appell polynomials of type II and give as consequence the power representation of a $q$-Appell polynomial sequence of type II whose determining function is known. In section 5, we prove that the only set of polynomials which is at the same time orthogonal and $q$-Appell of type II are the Al-Salam Carlitz II polynomials. In section 6, it is proved that the  polynomial sequences which are at the same time quasi-orthogonal and $q$-Appell of type II are  linear combinations of Al-Salam Carlitz II polynomials. Finally, in section 7, we provide a recursion formula and a $q$-difference equation for all the $q$-Appell polynomial sequences of type II.

\section{Preliminary definitions}

The following definitions can be found in \cite{KLS}. 
Let $n$ be a non-negative integer and define the so-called $q$-number  by 
\[  [n]_q=\dfrac{1-q^n}{1-q}.\]
For a non-negative integer $n$, the $q$-factorial is defined by 
\[[n]_q!=\prod_{k=0}^n[k]_{q}\quad\textrm{for}\quad n\geq1, \quad\textrm{and}\quad [0]_{q}!=1.\]
The $q$-binomial coefficients are defined by 
\[\qbinomial{n}{k}{q}=\dfrac{[n]_q!}{[k]_q![n-k]_q!},\quad (0\leq k\leq n).\]
The following so-called $q$-Pochhammer numbers $(a;q)_n$ are defined by 
\[(a;q)_0=1,\quad (a;q)_n=\prod_{k=0}^{n}(1-aq^k),\quad (n\geq 1).\]
It is not difficult to see that 
\[\qbinomial{n}{k}{q}=\dfrac{(q;q)_n}{(q;q)_k(q;q)_{n-k}},\quad (0\leq k\leq n).\]
We will use the following two $q$-analogues of the exponential function $e^x$:
\begin{equation}\label{smallqexp}
   e_q(x)=\sum_{k=0}^{\infty}\dfrac{x^k}{[k]_q!},
\end{equation}
and 
\begin{equation}\label{bigqexp}
  E_q(x)=\sum_{k=0}^{\infty}\dfrac{q^{\binom{k}{2}}}{[k]_q!}x^k,
\end{equation}
These two functions are related by the equation  (see \cite{kac})
\begin{equation}
   e_q(x)E_q(-x)=1.
\end{equation}

\noindent The following Cauchy product for infinite series applies 
\begin{equation}\label{cauchy1}
\left(\sum_{n=0}^{\infty}A_n\right)\left(\sum_{n=0}^{\infty}B_n\right)=\sum_{n=0}^{\infty}\left(\sum_{k=0}^n A_kB_{n-k}\right).
\end{equation} 
In particular, if $A_n=\dfrac{a_nx^n}{[n]_q!}$ and $B_n=\dfrac{b_nx^n}{[n]_q!}$, then we have 
\begin{equation}\label{cauchy2}
\left(\sum_{n=0}^{\infty}\dfrac{a_nx^n}{[n]_q!}\right)\left(\sum_{n=0}^{\infty}\dfrac{b_nx^n}{[n]_q!}\right)=\sum_{n=0}^{\infty}\left(\sum_{k=0}^n\qbinomial{n}{k}{q} a_kb_{n-k}\right)\dfrac{x^n}{[n]_q!}.
\end{equation} 

\section{Characterizations of $q$-Appell sets of type II}

\subsection{The characterization theorem}
\noindent In this section, we give several characterizations of  $q$-Appell sets of type II.
\begin{theorem}\label{theo-characterization}
Let $\{f_n(x)\}_{n=0}^{\infty}$ be a sequence of polynomials. Then the following are all equivalent: 
\begin{enumerate}
   \item $\{f_n(x)\}_{n=0}^{\infty}$ is a $q$-Appell set of type II. 
   \item There exists a sequence $(a_k)_{k\geq 0}$; independent of $n$; $a_0\neq 0$; such that 
   \[ f_n(x)=\sum_{k=0}^{n}\qbinomial{n}{k}{q}q^{\binom{n-k}{2}}a_{k}x^{n-k}.\]
   \item $\{f_n(x)\}_{n=0}^{\infty}$ is generated by 
   \[A(t)E_{q}(xt)=\sum_{n=0}^{\infty}f_n(x)\dfrac{t^n}{[n]_{q}!},\]
   where
   \begin{equation}
     A(t)=\sum_{n=0}^{\infty}a_n\dfrac{t^n}{[n]_{q}!},
   \end{equation}
   is called the {\sl determining function} for $\{f_n(x)\}_{n=0}^{\infty}$.
   \item There exists a sequence $(a_k)_{k\geq 0}$; independent of $n$; $a_0\neq 0$; such that
   \[ f_n(x)=\left(\sum_{k=0}^{\infty}\dfrac{a_k q^{\binom{n-k}{2}}}{[k]_{q}!}\Dq^{k}\right)x^n.\]
\end{enumerate}
\end{theorem}

\begin{proof}
First, we prove that $(1)\implies (2)\implies (3)\implies (1)$.
\begin{description}
   \item[$(1)\implies (2)$.] Since $\{f_n(x)\}_{n=0}^{\infty}$ is a polynomial set, it is possible to write 
   \begin{equation}\label{exp1}
      f_n(x)=\sum_{k=0}^{n}a_{n,k}\qbinomial{n}{k}{q}q^{\binom{n-k}{2}}x^{n-k},\quad n=1,2,\ldots,
   \end{equation}
   where the coefficients $a_{n,k}$ depend on $n$ and $k$ and $a_{n,0}\neq 0$. We need to prove that these coefficients are independent of $n$. By  applying the operator $\Dq$ to each member of \eqref{exp1} and taking into account that $\{f_n(x)\}_{n=0}^{\infty}$ is a $q$-Appell polynomial set of type II, we obtain 
   \begin{equation}\label{exp2}
      f_{n-1}(qx)=\sum_{k=0}^{n-1}a_{n,k}\qbinomial{n-1}{k}{q}q^{\binom{n-1-k}{2}}(qx)^{n-1-k}, \quad n=1,2,\ldots,
   \end{equation}
   since $\Dq x^{0}=0$. Shifting index $n\to n+1$ in \eqref{exp2} and making the substitution $x\to xq^{-1}$, we get 
   \begin{equation}\label{exp3}
      f_n(x)=\sum_{k=0}^{n}a_{n+1,k}\qbinomial{n}{k}{q}q^{\binom{n-k}{2}}x^{n-k},\quad n=0,1,\ldots,
   \end{equation}
   Comparing \eqref{exp1} and \eqref{exp3}, we have $a_{n+1,k}=a_{n,k}$ for all $k$ and $n$, which means that $a_{n,k}=a_k$ is independent of $n$. 
  \item[$(2)\implies (3)$.] From $(2)$, and the identity \eqref{cauchy2}, we have 
  \begin{eqnarray*}
  \sum_{n=0}^{\infty}f_n(x)\dfrac{t^n}{[n]_{q}!}&=& \sum_{n=0}^{\infty}\left(\sum_{k=0}^{n}\qbinomial{n}{k}{q}q^{\binom{n-k}{2}}a_{k}x^{n-k}\right)\dfrac{t^n}{[n]_{q}!}\\
   &=& \left(\sum_{n=0}^{\infty}a_n \dfrac{t^n}{[n]_{q}!}\right)\left(\sum_{n=0}^{\infty}\dfrac{q^{\binom{n}{2}}}{[n]_{q}!}(xt)^n\right)\\
   &=& A(t)E_{q}(xt). 
   \end{eqnarray*}
   \item[$(3)\implies (1)$.] Assume that $\{f_n(x)\}_{n=0}^{\infty}$ is generated by 
   \[A(t)E_{q}(xt)=\sum_{n=0}^{\infty}f_n(x)\dfrac{t^n}{[n]_{q}!}.\] 
   \noindent Then, applying the operator $\Dq$ to each side of this equation,
   \begin{eqnarray*}
    tA(t)E_{q}(qxt)&=& \sum_{n=0}^{\infty}\Dq f_n(x)\dfrac{t^n}{[n]_{q}!}.
   \end{eqnarray*}
   Moreover, we have 
   \begin{eqnarray*}
     tA(t)E_{q}(qxt)&=& \sum_{n=0}^{\infty}f_{n}(qx)\dfrac{t^{n+1}}{[n]_{q}!}=\sum_{n=0}^{\infty}[n]_{q}f_{n-1}(qx)\dfrac{t^n}{[n]_{q}!}.
   \end{eqnarray*}
   By comparing the coefficients of $t^n$, we obtain $(1)$. 
\end{description}
Next, $(2)\iff (4)$ is obvious. 
This ends the proof of the theorem.
\end{proof}

\subsection{Some examples of $q$-Appell polynomials of type II}

\noindent In this section we provide several examples of $q$-Appell polynomials of type II. 

\subsubsection{$q$-Bernoulli polynomials of type II}

\noindent We first introduce  new $q$-analogue of Bernoulli numbers.

\begin{definition}\label{def10}
The one-parameter $q$-Bernoulli numbers of type II $\Bn{\alpha}{n}$ are defined by 
\begin{eqnarray}
    \left(\dfrac{t}{E_q(t)-1}\right)^{\alpha}&=&\sum_{n=0}^{\infty}\Bn{\alpha}{n}\dfrac{t^n}{[n]_q!}, \label{def-gen-q-ber-numb}
    \end{eqnarray}
When $\alpha=1$, we have the $q$-Bernoulli numbers  of type II $\Bern{n}$ given by 
\begin{eqnarray}
    \dfrac{t}{E_q(t)-1}&=&\sum_{n=0}^{\infty}\Bern{n}\dfrac{t^n}{n!}. \label{def-q-ber-numb}
\end{eqnarray}
\end{definition}

\begin{definition}
The one-parameter $q$-Bernoulli polynomials of type II $\Bn{\alpha}{n}(x)$ are defined by 
\begin{eqnarray}
    \left(\dfrac{t}{E_q(t)-1}\right)^{\alpha}E_q(xt)&=&\sum_{n=0}^{\infty}\Bn{\alpha}{n}(x)\dfrac{t^n}{n!},\quad \label{def-gen-q-ber-pol}
    \end{eqnarray}
When $\alpha=1$, we have the $q$-Bernoulli polynomials  of type II $\Bern{n}$ given by 
\begin{eqnarray}
    \dfrac{t}{E_q(t)-1}E_q(xt)&=&\sum_{n=0}^{\infty}\Bern{n}(x)\dfrac{t^n}{n!},\quad  \label{def-q-ber-pol}
\end{eqnarray}
\end{definition}

\begin{theorem}
The polynomials $\Bn{\alpha}{n}(x)$ have the representation
\begin{equation}\label{represent-qbernoulli2}
\Bn{\alpha}{n}(x)=\sum_{k=0}^{n}\qbinomial{n}{k}{q}q^{\binom{k}{2}}\Bn{\alpha}{n-k}x^k.
\end{equation}
\end{theorem}

\begin{proof}
From the generating functions \eqref{def-gen-q-ber-numb} and \eqref{def-gen-q-ber-pol}, we get 
\begin{eqnarray*}
    \left(\dfrac{t}{E_q(t)-1}\right)^{\alpha}E_q(xt)&=&\left(\sum_{n=0}^{\infty}\Bn{\alpha}{n}\dfrac{t^n}{[n]_q!}\right)\left(\sum_{n=0}^{\infty}q^{\binom{n}{2}}x^n\dfrac{t^n}{[n]_q!}\right)\\
    &=&  \sum_{n=0}^{\infty}\left(\sum_{k=0}^{n}\qbinomial{n}{k}{q}q^{\binom{k}{2}}\Bn{\alpha}{n-k}x^k  \right)\dfrac{t^n}{[n]_q!}
\end{eqnarray*}
Comparing the coefficients of $t^n$ the representation follows.  
\end{proof}

\begin{remark}
Note that \eqref{represent-qbernoulli2} could be obtained directly using Theorem \ref{theo-characterization} by taking $a_k=\Bn{\alpha}{k}$. 
\end{remark}

\begin{theorem}
The polynomials $\Bn{\alpha}{n}(x)$ have the following power representation
\begin{equation*}
x^n=q^{-\binom{n}{2}}\sum_{k=0}^{n}\dfrac{q^{\binom{k+1}{2}}}{[k+1]_q}\qbinomial{n}{k}{q}\Bern{n-k}(x).
\end{equation*}
\end{theorem}

\begin{proof}
The proof follows from the generating function \eqref{def-q-ber-pol}. For a general proof, we refer the reader to Proposition \ref{propo-inv}.
\end{proof}

\subsubsection{$q$-Euler polynomials of type II}

\noindent We first introduce  new $q$-analogue  of Euler numbers.

\begin{definition}\label{def11}
The one-parameter $q$-Euler numbers of type II $\En{\alpha}{n}$ are defined by 
\begin{eqnarray}
    \left(\dfrac{2}{E_q(t)+1}\right)^{\alpha}&=&\sum_{n=0}^{\infty}\En{\alpha}{n}\dfrac{t^n}{n!},\quad \label{def-gen-q-eul-numb}
    \end{eqnarray}
When $\alpha=1$, we have the $q$-Euler numbers of type II $\Euln{n}$ given by 
\begin{eqnarray}
    \dfrac{2}{E_q(t)+1}&=&\sum_{n=0}^{\infty}\Euln{n}\dfrac{t^n}{n!},\quad  \label{def-disc-q-eul-numb}
\end{eqnarray}
\end{definition}

\begin{definition}
The one-parameter $q$-Euler polynomials of type II $\En{\alpha}{n}(x)$ is defined by 
\begin{eqnarray}
    \left(\dfrac{2}{E_q(t)+1}\right)^{\alpha}E_q(xt)&=&\sum_{n=0}^{\infty}\En{\alpha}{n}(x)\dfrac{t^n}{n!},\quad \label{def-gen-q-eul-pol}
    \end{eqnarray}
When $\alpha=1$, we have the $q$-Euler polynomials  of type II $\Euln{n}$ given by 
\begin{eqnarray}
    \dfrac{2}{E_q(t)+1}E_q(xt)&=&\sum_{n=0}^{\infty}\Euln{n}(x)\dfrac{t^n}{n!},\quad \label{def-q-eul-pol}
\end{eqnarray}
\end{definition}

\begin{theorem}
The polynomials $\En{\alpha}{n}(x)$  have the representation
\begin{equation}
\En{\alpha}{n}(x)=\sum_{k=0}^{n}\qbinomial{n}{k}{q}q^{\binom{k}{2}}\En{\alpha}{n-k}x^k.
\end{equation}
\end{theorem}

\begin{proof}
From the generating functions \eqref{def-gen-q-eul-numb} and \eqref{def-gen-q-eul-pol}, we get 
\begin{eqnarray*}
    \left(\dfrac{2}{E_q(t)+1}\right)^{\alpha}E_q(xt)&=&\left(\sum_{n=0}^{\infty}\En{\alpha}{n}\dfrac{t^n}{[n]_q!}\right)\left(\sum_{n=0}^{\infty}q^{\binom{n}{2}}x^n\dfrac{t^n}{[n]_q!}\right)\\
    &=&  \sum_{n=0}^{\infty}\left(\sum_{k=0}^{n}\qbinomial{n}{k}{q}q^{\binom{k}{2}}\En{\alpha}{n-k}x^k  \right)\dfrac{t^n}{[n]_q!}
\end{eqnarray*}
Comparing the coefficients of $t^n$ the representation follows.  
\end{proof}

\begin{theorem}
The polynomials $\Euln{n}(x)$ have the power representation 
\begin{equation*}
x^n=\dfrac{1}{2q^{\binom{n}{2}}}\left(\Euln{n}(x)+\sum_{k=0}^{n}\qbinomial{n}{k}{q}q^{\binom{k}{2}}\Euln{n-k}(x)\right).
\end{equation*}
\end{theorem}

\begin{proof}
The proof follows from the generating function \eqref{def-q-eul-pol}.
\end{proof}

\subsubsection{Modified Al-Salam Carlitz II polynomials}

\noindent The Al-Salam Carlitz II polynomials $V_n^{(a)}(x;q)$  \cite[P.\ 538]{KLS}
 fulfil the $q$-difference equation 
\begin{equation}
D_q V_{n}^{(a)}(x;q)=q^{-n+1}[n]_qV_{n-1}^{(a)}(qx;q).
\end{equation}
Let us define the modified Al-Salam Carlitz II polynomials $\mathcal{V}_n^{(a)}(x;q)$ by the relation 
\begin{equation}\label{modi-al-s2}
 \mathcal{V}_n^{(a)}(x;q)= q^{\binom{n}{2}}{V}_n^{(a)}(x;q).
\end{equation}
Then we have the following proposition.
\begin{proposition}
The polynomial sequence $\{\mathcal{V}_n^{(a)}(x;q)\}_{n=0}^{\infty}$ is a $q$-Appell polynomial set of type II. 
\end{proposition}

\section{Algebraic structure}

\noindent  We denote a given polynomial set $\{f_n(x)\}_{n=0}^{\infty}$ by a single symbol $f$ and refer to $f_n(x)$ as the $n$-th component of $f$. We define (see \cite{appell1880, sheffer1931}) on the set $\mathcal{P}$ of all polynomial sets the following  operation $+$. This operation is given by the rule that $f+g$ is the polynomial set whose $nth$ component is $f_n(x)+g_n(x)$ provided that the degree of $f_n(x)+g_n(x)$ is exactly $n$. We also define the operation $*$ (which appears here for the fist time) such that if $f$ and $g$ are two sets whose $nth$ components are, respectively, 
\[f_n(x)=\sum_{k=0}^n\alpha(n,k)x^k,\quad g_n(x)=\sum_{k=0}^n\beta(n,k)x^k,\]
then $f*g$ is the polynomial set whose $nth$ component is 
\[ (f*g)_n(x)=\sum_{k=0}^n\alpha(n,k)q^{-\binom{k}{2}}g_k(x).\]
If $\lambda$ is a real or complex number, then $\lambda f$ is defined as the polynomial set whose $nth$ component is $\lambda f_n(x)$. We obviously have 
\begin{eqnarray*}
f+g=g+f\quad \textrm{for all}\quad f,g\in \mathcal{P},\\
\lambda f*g=(f*\lambda g)=\lambda(f*g).
\end{eqnarray*}
Clearly, the operation $*$ is not commutative on $\mathcal{P}$. One commutative subclass is the set $\A$ of all Appell polynomials (see \cite{appell1880}). 

In what follows, $\A(q)$ denotes the class of all $q$-Appell sets of type II. 

In $\A(q)$ the identity element (with respect to $*$) is the $q$-Appell set of type II $\mathcal{I}=\left\{q^{\binom{n}{2}}x^n\right\}$. Note that $\mathcal{I}$ has the determining function $A(t)=1$. This is due to the identity \eqref{bigqexp}. The following theorem is easy to prove. 

\begin{theorem}
Let $f,\, g,\, h\in\A(q)$ with the determining functions $A(t)$, $B(t)$ and $C(t)$ respectively. Then 
\begin{enumerate}
   \item $f+g\in\A(q)$ if $A(0)+B(0)\neq 0$,
   \item $f+g$ belongs to the determining function $A(t)+B(t)$,
   \item $f+(g+h)=(f+g)+h$.
\end{enumerate}
\end{theorem} 
The next theorem is less obvious. 

\begin{theorem}\label{theo-group}
If $f,\, g,\, h\in\A(q)$ with the determining functions $A(t)$, $B(t)$ and $C(t)$ respectively, then
\begin{enumerate}
    \item $f*g\in\A(q)$
    \item $f*g=g*f$,
    \item $f*g$ belongs to the determining function $A(t)B(t)$,
    \item $f*(g*h)=(f*g)*h$. 
\end{enumerate}
\end{theorem}
\begin{proof}
It is enough to prove the first part of the theorem. The rest follows directly. \\
According to Theorem \ref{theo-characterization}, we may put 
\[ f_n(x)=\sum_{k=0}^{n}\qbinomial{n}{k}{q}q^{\binom{n-k}{2}}a_{k}x^{n-k}=\sum_{k=0}^{n}\qbinomial{n}{k}{q}q^{\binom{k}{2}}a_{n-k}x^{k}\]
so that 
 \[A(t)=\sum_{n=0}^{\infty}a_n\dfrac{t^n}{[n]_{q}!}.\]
 Hence
 \begin{eqnarray*}
 \sum_{n=0}^{\infty}(f*g)_n(x)\dfrac{t^n}{[n]_{q}!}&=& \sum_{n=0}^{\infty}\left(\sum_{k=0}^n\qbinomial{n}{k}{q}a_{n-k}g_k(x)  \right) \dfrac{t^n}{[n]_{q}!}\\
 &=& \left(\sum_{n=0}^{\infty}a_n\dfrac{t^n}{[n]_{q}!}\right)\left(\sum_{n=0}^{\infty}g_n(x)\dfrac{t^n}{[n]_{q}!}\right)\\
 &=& A(t)B(t)E_{q}(xt).
 \end{eqnarray*}
 This ends the proof of the theorem. 
\end{proof}

\begin{corollary}\label{coro-inv}
Let $f\in\A(q)$ then there is a set $g\in\A(q)$ such that 
\[f*g=g*f=\mathcal{I}.\]
\end{corollary}
Indeed $g$ belongs to the determining function $(A(t))^{-1}$ where $A(t)$ is the determining function for $f$. 

In view of Corollary \ref{coro-inv} we shall denote this element $g$ by $f^{-1}$. We are further motivated by Theorem \ref{theo-group} and its corollary to define $f^{0}=\mathcal{I}$, $f^{n}=f*(f^{n-1})$ where $n$ is a non-negative integer, and $f^{-n}=f^{-1}*(f^{-n+1})$. We note that we have proved that the system $(\A(q),*)$ is a commutative group. In particular this leads to the fact that if 
\[f * g=h\]
and if any two of the elements $f,\, g,\, h$ are $q$-Appell of type II then the third is also $q$-Appell of type II. 

\begin{proposition}\label{propo-inv}
If $f$ is a $q$-Appell set of type II with the determining function $A(t)$, if we put 
\[A^{-1}(t)=\sum_{n=0}^{\infty}b_n\dfrac{t^n}{[n]_{q}!}\]
then 
\[x^n=q^{-\binom{n}{2}}\sum_{k=0}^{n}\qbinomial{n}{k}{q}b_kf_{n-k}(x).\]
\end{proposition}

\begin{proof}
Since $f$ is a $q$-Appell set of type II, we have 
\begin{eqnarray*}
   \sum_{n=0}^{\infty}q^{\binom{n}{2}}x^n\dfrac{t^n}{[n]_{q}!}&=& (A(t))^{-1}A(t)E_{q}(xt)\\
   &=& \left(\sum_{n=0}^{\infty}b_n\dfrac{t^n}{[n]_{q}!}\right)\left(\sum_{n=0}^{\infty}f_n(x)\dfrac{t^n}{[n]_{q}!}\right)\\
   &=& \sum_{n=0}^{\infty}\left(\sum_{k=0}^{n}\qbinomial{n}{k}{q}b_kf_{n-k}(x)\right)\dfrac{t^n}{[n]_{q}!}.
\end{eqnarray*}
The result follows by comparing the coefficients of $t^n$. 
\end{proof}

\section{Orthogonal $q$-Appell polynomials of type II}

In this section we determine those real sets in $\A(q)$ which are also orthogonal. It is well known \cite{szego} that a set of real orthogonal polynomials satisfies a recurrence relation of the form
\begin{equation}\label{ttrr}
P_{n+1}(x)=(A_nx+B_n)P_n(x)+C_n P_{n-1}(x),\quad n\geq 1,
\end{equation}
with 
\[P_0(x)=1,\quad P_1(x)=A_0x+B_0.\]
Here $A_n$, $B_n$ and $C_n$ are real constants which do not depend on $n$. 

If we $q$-differentiate \eqref{ttrr} and assume that the polynomial set $\{P_n(x)\}$ is $q$-Appell of type II, we get: 
\begin{equation}\label{step1}
[n+1]_qP_n(qx)=[n]_q\left(A_n x+B_n\right)P_{n-1}(qx)+A_n P_n(qx)+[n-1]_qC_nP_{n-2}(qx).
\end{equation}
Substituting $n$ by $n+1$ and $x$ by $xq^{-1}$ in  \eqref{step1}, it follows that
\begin{equation}\label{step2} 
P_{n+1}(x)=\left(\dfrac{[n+1]_qq^{-1}A_{n+1}}{[n+2]_q-A_{n+1}}x+\dfrac{[n+1]_qB_{n+1}}{[n+2]_q-A_{n+1}}\right)P_n(x)+\dfrac{[n]_qC_{n+1}}{[n+2]_q-A_{n+1}}P_{n-1}(x).
\end{equation}
By comparing \eqref{ttrr} and \eqref{step2} we get 
\[\dfrac{[n+1]_qA_{n+1}}{[n+2]_q-A_{n+1}}=qA_{n},\quad \dfrac{[n+1]_qB_{n+1}}{[n+2]_q-A_{n+1}}=B_n\quad\textrm{and}\quad \dfrac{[n]_qC_{n+1}}{[n+2]_q-A_{n+1}}=C_n,\]
so that 
\[A_n=q^n,\quad B_n=B_0\quad\textrm{and}\quad C_n=C_1(1-q^n).\]
Hence, $\{P_n(x)\}$ is given by 
\begin{equation}\label{fin-step}
P_{n+1}(x)=(q^nx+B_0)P_n(x)+C_1(1-q^n)P_{n-1}(x),
\end{equation}
\[P_0(x)=1,\quad P_1(x)=x+B_0.\]
From the recurrence relation of the Al-Salam Carlitz II polynomials (see \cite[P.\ 538]{KLS}), it is easy to see that the polynomial sequence $\{R_n(x)\}$ with
\[
R_n(x)=\beta^n q^{\binom{n}{2}}V_{n}^{(\frac{\alpha}{\beta})}\left(\frac{x}{\beta};q\right)
\]
satisfies the recurrence relation 
\begin{equation}\label{recu-Rn}
xR_n(x)=R_{n+1}(x)+(q^nx-(\alpha+\beta))R_n(x)-\alpha\beta (1-q^n)R_{n-1}(x)
\end{equation}
with $R_0(x)=1$ and $R_1(x)=x-(\alpha+\beta)$. 
It is therefore clear that 
\textcolor{black}{
\begin{equation}\label{eq-fin}
P_n(x)=\beta^n q^{\binom{n}{2}}V_{n}^{(\frac{\alpha}{\beta})}\left(\frac{x}{\beta};q\right).
\end{equation}}
where $\alpha+\beta=-B_0$ and $\alpha\beta=-C_1$. 

We thus have the following theorem.
\begin{theorem}\label{th1}
The set of $q$-Appell polynomials of type II which are also orthogonal is given \eqref{fin-step} or \eqref{eq-fin}.
\end{theorem}

%
%
%

\section{Quasi-orthogonal $q$-Appell polynomials of type II}

\noindent A sequence of polynomials $\{Q_n(x)\}$, $n=0,\; 1,\; 2,\; \ldots $, $\deg Q_n(x)=n$ is said to be quasi-orthogonal if there is an interval $(a,b)$ and a non-decreasing function $\alpha(x)$ such that 
\[ \int_a^b x^mQ_n(x)d\alpha(x)\left\{\begin{array}{ll}
=0 & \textrm{for}\quad 0\leq m\leq n-2\\
\neq 0 &\textrm{for}\quad 0\leq m=n-1\\
\neq 0 &\textrm{for}\quad 0=m=n. 
\end{array}\right.\]
We say that two polynomial sets are related if one set is quasi-orthogonal with respect to the interval and the distribution of the orthogonality of the other set.
Riesz \cite{riesz} and Chihara \cite{chihara} have shown that a necessary and sufficient condition for the quasi-orthogonality of the $\{Q_n(x)\}$ is that there exist nonzero constants, $\{a_n\}_{n=0}^{\infty}$ and $\{b_n\}_{n=1}^{\infty}$, such that 
\begin{equation}
   \begin{array}{lll}
      Q_n(x)&=&a_n P_n(x)+b_n P_{n-1}(x),\\
      Q_0(x)&=& a_0P_0(x)
   \end{array}\quad n\geq 1,
\end{equation}
where the $\{P_n(x)\}_{n=0}^{\infty}$ are the related orthogonal polynomials.

The following two propositions are of particular interest. 

\begin{proposition}{\em(See \cite[Theorem 1]{dickinson})}\label{dickinson}
For $\{Q_n(x)\}$ to be a set of polynomials quasi-orthogonal with respect to an interval $(a,b)$ and a distribution $d\alpha(x)$, it is necessary and sufficient that there exist a set of nonzero constants $\{T_k\}_{k=0}^{\infty}$ and a set of polynomials $\{P_n(x)\}$ orthogonal with respect to $(a,b)$ and $d\alpha(x)$ such that 
\begin{equation}\label{connect0}
  P_n(x)=\sum_{k=0}^{n}T_kQ_k(x),\quad n\geq 0.
\end{equation}
\end{proposition}

\begin{proposition}{\em(See \cite[Theorem 2]{dickinson})}\label{dick2}
A necessary and sufficient condition that the set $\{Q_n(x)\}_{n=0}^{\infty}$ where each $Q_n(x)$ is a polynomial of degree precisely $n$, be quasi-orthogonal is that it satisfies 
\begin{equation*}
Q_{n+1}(x)=(A_nx+B_n)Q_n(x)+C_nQ_{n-1}(x)+E_n\sum_{k=0}^{n-2}T_kQ_k(x),
\end{equation*}
for all $n$, with $E_0=E_1=0$. 
\end{proposition}

\begin{theorem}\label{theo-car1}
If $\{Q_n(x)\}_{n=0}^{\infty}$ is a $q$-Appell set of type II of quasi-orthogonal polynomials,
 then there exist three reel numbers $B_0$, $C_1$ and $\lambda$, such that 
\begin{equation}\label{rec-car1}
 Q_{n+1}(x)=(q^nx+B_0)Q_n(x)+C_1(1-q^{n})Q_{n-1}(x)+\dfrac{[n]_q!}{\lambda^n}\sum_{k=0}^{n-2}\dfrac{\lambda^k}{[k]_q!}Q_k(x). 
\end{equation}
\end{theorem}

\begin{proof}
Assume that $\{Q_n(x)\}_{n=0}^{\infty}$ is a $q$-Appell set which is quasi-orthogonal and $\{P_n(x)\}_{n=0}^{\infty}$ the related orthogonal family. From Proposition \ref{dick2}, there exist four sequences $\{A_n\}_{n=0}^{\infty}$, $\{B_n\}_{n=0}^{\infty}$, $\{C_n\}_{n=0}^{\infty}$ and $\{E_n\}_{n=0}^{\infty}$ with $E_0=E_1=0$ such that 
\begin{equation}\label{rr-quasi}
Q_{n+1}(x)=(A_nx+B_n)Q_n(x)+C_nQ_{n-1}(x)+E_n\sum_{k=0}^{n-2}T_kQ_k(x).
\end{equation}
If we $q$-differentiate \eqref{rr-quasi} and use the fact that $\{Q_n(x)\}_{n=0}^{\infty}$ is a $q$-Appell set of type II, we get after some simplifications 
\begin{eqnarray} 
Q_{n+1}(x)&=&\left(\dfrac{[n+1]_qq^{-1}A_{n+1}}{[n+2]_q-A_{n+1}}x+\dfrac{[n+1]_qB_{n+1}}{[n+2]_q-A_{n+1}}\right)Q_n(x)\nonumber\\
&&+\dfrac{[n]_qC_{n+1}}{[n+2]_q-A_{n+1}}Q_{n-1}(x)+\dfrac{E_{n+1}}{[n+2]_q-A_{n+1}}\sum_{k=0}^{n-2}[k+1]_qT_{k+1}Q_k(x).\label{rr-quasi2}
\end{eqnarray}
By comparing \eqref{rr-quasi} and \eqref{rr-quasi2} we get 
\[A_n=q^n,\quad B_n=B_0\quad\textrm{and}\quad C_n=C_1(1-q^n),\]
and 
\[E_nT_k=\dfrac{E_{n+1}[k+1]_qT_{k+1}}{[n+2]_q-A_{n+1}}=\dfrac{[k+1]_qT_{k+1}}{[n+1]_q}E_{n+1},\]
For $k=0$ and $k=n-2$, we obtain the following 
\begin{equation}\label{coef3}
E_{n+1}=\dfrac{[n+1]_q}{T_1}E_n,\qquad T_{n}=\dfrac{E_{n+1}}{E_{n+2}}\dfrac{[n+2]_q}{[n]_q}T_{n-1}.
\end{equation}
If, for a given $k\geq 2$, $E_k=0$, it follows from \eqref{coef3} that $E_k=0$ for all $k$. In this case \eqref{rr-quasi} becomes a three-term recurrence relation 
\begin{equation}\label{tt-rr2}
Q_{n+1}(x)=(A_nx+B_n)Q_n(x)+C_nQ_{n-1}(x).
\end{equation} 
In this case, from Theorem \ref{th1}, it is seen that $\{Q_n(x)\}_{n=0}^{\infty}$ is essentially the sequence of Al-Salam Carlitz II polynomials. Thus, in this case, $\{Q_n(x)\}_{n=0}^{\infty}$ is not a sequence of quasi-orthogonal polynomials. Thus, we must have $E_k\neq 0$ for $k\geq 2$. 

Again, using \eqref{coef3}, we have for all $n\geq 0$ the identities $E_n=\dfrac{[n]_q!}{T_1^n}$ and  $\dfrac{T_{n-1}}{[n]_qT_n}=\dfrac{1}{T_1}$. This last relation gives 
$T_n=\dfrac{T_1^n}{[n]_q!}$. Seting  $T_1=\lambda$, this ends the proof of the theorem. 
\end{proof}

\begin{theorem}\label{theo-car2}
Let $\{Q_n(x)\}_{n=0}^{\infty}$ be a polynomial set. The following assertions are equivalent:
\begin{enumerate}
   \item $\{Q_n(x)\}_{n=0}^{\infty}$ is quasi-orthogonal and is a $q$-Appell set of type II.
   \item There exists three constants $\alpha$, $\beta$ and $\gamma$ ($\beta,\gamma\neq 0$) such that 
   \[ Q_n(x)=\beta^n q^{\binom{n}{2}}V_{n}^{(\frac{\alpha}{\beta})}\left(\frac{x}{\beta};q\right)-\dfrac{\beta^{n-1} q^{\binom{n-1}{2}}[n]_q!}{\lambda^n}V_{n-1}^{(\frac{\alpha}{\beta})}\left(\frac{x}{\beta};q\right)\;\quad (n\geq 1),\]
   where $V_n^{(a)}(x;q)$ are the Al-Salam Carlitz II polynomials.  
\end{enumerate}
\end{theorem}

\begin{proof}
Suppose first that $\{Q_n(x)\}_{n=0}^{\infty}$ is quasi-orthogonal and is a $q$-Appell set of type II. Then, by Theorem \ref{theo-car1}, the $Q_n$'s satisfy a recurrence relation of the form \eqref{rec-car1}. Let us define the polynomial set $\{P_n(x)\}_{n=0}^{\infty}$ by
\begin{equation}\label{rel-pol}
P_n(x)=\dfrac{[n]_q!}{\lambda^n}\sum_{k=0}^{n}\dfrac{\lambda^k}{[k]_q!}Q_k(x). 
\end{equation} 
It is not difficult to see that 
\[ D_q P_n(x)=\dfrac{[n]_q!}{\lambda^n}\sum_{k=1}^{n}\dfrac{\lambda^k}{[k]_q!}[k]_qQ_{k-1}(qx)= [n]_q\dfrac{[n-1]_q!}{\lambda^{n-1}}\sum_{k=0}^{n-1}\dfrac{\lambda^k}{[k]_q!}Q_k(qx)=[n]_q P_{n-1}(qx).\]
Hence, $\{P_n(x)\}_{n=0}^{\infty}$ is a $q$-Appell set of type II. Moreover, $\{P_n(x)\}_{n=0}^{\infty}$ is the orthogonal set (see Proposition \ref{dickinson}) related to $\{Q_n(x)\}_{n=0}^{\infty}$. By Theorem \ref{th1}, there exist $\alpha$ and $\beta$ such that 
\[P_n(x)=\beta^n q^{\binom{n}{2}}V_{n}^{(\frac{\alpha}{\beta})}\left(\frac{x}{\beta};q\right).\]
 Next, from \eqref{rel-pol}, it follows easily that 
\begin{eqnarray*}
Q_n(x)&=&P_n(x)-\dfrac{[n]_q!}{\lambda^n}P_{n-1}(x)\\
&=& \beta^n q^{\binom{n}{2}}V_{n}^{(\frac{\alpha}{\beta})}\left(\frac{x}{\beta};q\right)-\dfrac{\beta^{n-1} q^{\binom{n-1}{2}}[n]_q!}{\lambda^n}V_{n-1}^{(\frac{\alpha}{\beta})}\left(\frac{x}{\beta};q\right)
\end{eqnarray*}
The first implication of the theorem follows. \\
Conversely, assume that  there exist three constants $\alpha$, $\beta$ and $\gamma$ ($\beta,\gamma\neq 0$) such that 
 \[ Q_n(x)=\beta^n q^{\binom{n}{2}}V_{n}^{(\frac{\alpha}{\beta})}\left(\frac{x}{\beta};q\right)-\dfrac{\beta^{n-1} q^{\binom{n-1}{2}}[n]_q!}{\lambda^n}V_{n-1}^{(\frac{\alpha}{\beta})}\left(\frac{x}{\beta};q\right)\;\quad (n\geq 1).\]
   It is easy to see that $\{Q_n(x)\}_{n=0}^{\infty}$ is a quasi-orthogonal set. It remains to prove that $\{Q_n(x)\}_{n=0}^{\infty}$ is a $q$-Appell set. Using the fact that $D_q [f(a x)]=a [D_qf](ax)$, we have 
   \[D_q V_n^{(\alpha/\beta)}\left(\dfrac{x}{\beta};q\right)=\dfrac{[n]_qq^{-n+1}}{\beta} V_{n-1}^{(\alpha/\beta)}\left(\dfrac{qx}{\beta};q\right).\] It follows that $D_q Q_n(x)=[n]_q Q_{n-1}(qx)$. This ends the proof of the theorem.
\end{proof}

\section{Recursion formula and $q$-difference equation}

\noindent In this section, we derive a recurrence relation and a $q$-difference equation for the $q$-Appell polynomials of type II. 

\begin{theorem}\label{recursion1}
Let $\{f_n(x)\}_{n=0}^{\infty}$ be the $q$-Appell polynomial sequence generated by 
\[A(t)E_q(xt)=\sum_{n=0}^{\infty}f_n(x)\dfrac{t^n}{[n]_q!}.\]
Then the following linear homogeneous recurrence relation  holds true: 
\begin{eqnarray}\label{recursion-formula1}
f_n(x/q)=\dfrac{1}{[n]_q}\sum_{k=0}^{n}\qbinomial{n}{k}{q}\alpha_kf_{n-k}(x)+{q^{-1}}{x}f_{n-1}(x).
\end{eqnarray}
where 
\[t\dfrac{D_{q,t}A(t)}{A(t)}=\sum_{n=0}^{\infty}\alpha_n\dfrac{t^n}{[n]_q!}.\]
\end{theorem}

\begin{proof}
If we $q$-differenciate each side of the generating function with respect to $t$ and multiply  the obtained equation by $t$, we get the following equations 
\begin{eqnarray*}
tD_{q,t}\left(A(t)E_q(xt)\right)&=&t\left[A(t)D_{q,t}E_q(xt)+D_{q,t}A(t) E_{q}(qxt)\right]\\
&=& t\left[xA(t)E_q(qxt)+D_{q,t}A(t)E_q(qxt)\right]\\
&=& A(t)E_{q}(qxt)\left(t\dfrac{D_{q,t}A(t)}{A(t)}+tx\right)\\
&=& \left(\sum_{n=0}^{\infty}f_n(qx)\dfrac{t^n}{[n]_q!}\right)\left(\sum_{n=0}^{\infty}\alpha_n\dfrac{t^n}{[n]_{q}!}+tx\right)\\
&=& \sum_{n=0}^{\infty}\left(\sum_{k=0}^{n}\qbinomial{n}{k}{q}\alpha_kf_{n-k}(qx)+[n]_qf_{n-1}(qx)\right)\dfrac{t^n}{[n]_q!}
\end{eqnarray*}
and 
\[tD_{q,t}\left(A(t)E_q(xt)\right)=t\sum_{n=0}^{\infty}[n]_qf_n(x)\dfrac{t^{n-1}}{[n]_q!}=\sum_{n=0}^{\infty}[n]_qf_n(x)\dfrac{t^{n}}{[n]_q!}.\]
Equating the coefficients of $t^n$, we obtain 
\[[n]_qf_n(x)=\sum_{k=0}^{n}\qbinomial{n}{k}{q}\alpha_kf_{n-k}(qx)+[n]_qxf_{n-1}(qx).\]
Substituting $x$ by $xq^{-1}$ we get the result.
\end{proof}

\noindent In order to state the $q$-difference equation, we need the following Lemma. 

\begin{lemma}
Define the $q$-shifted operator by 
\[\mathcal{E}^n_qf(x)=f(q^nx),\quad (n\in\Z).\]
Then, the following relation applies
\begin{equation}\label{commut1}
 \mathcal{E}^{-n}_q D_q^n= q^{\binom{n}{2}}D_{q^{-1}}^n, \quad n=0,1,2,\ldots
\end{equation}
\end{lemma}

\begin{proof}
The relation is obvious for $n=0$ and $n=1$. Let $n\geq 1$, assume that \eqref{commut1} holds true. Then 
\begin{eqnarray*}
\mathcal{E}_q^{-(n+1)}D_q^{n+1}= \mathcal{E}_q^{-1}\left(\mathcal{E}_q^{-n}D_q^{n}\right) D_q=q^{\binom{n}{2}}\mathcal{E}_q^{-1}D_{q^{-1}}^nD_q= q^{\binom{n}{2}}q^nD_{q^{-1}}^n\mathcal{E}_q^{-1}D_q=q^{\binom{n+1}{2}}D_{q^{-1}}^{n+1}.
\end{eqnarray*}
Note that the relation 
$D_{q^{-1}}D_q =qD_qD_{q^{-1}}$ has been used. 
\end{proof}

Now, we are able to prove the following theorem. 

\begin{theorem}
Let $\{f_n(x)\}_{n=0}^{\infty}$ be the $q$-Appell polynomial sequence of type II generated by 
\[A(t)E_q(xt)=\sum_{n=0}^{\infty}f_n(x)\dfrac{t^n}{[n]_q!}.\]
Assume that 
\[t\dfrac{D_{q,t}A(t)}{A(t)}=\sum_{n=0}^{\infty}\alpha_n\dfrac{t^n}{[n]_q!},\]
is valued around the point $t=0$. Then the $f_n$'s satisfy the $q$-difference equation 
\[\sum_{k=0}^{n}\dfrac{q^{\binom{k}{2}}}{[k]_q!}\alpha_k D_{q^{-1}}^k f_n(x)+\dfrac{x}{q} D_{q^{-1}}f_{n}(x)-[n]_qf_n(x/q)=0.\]
\end{theorem}

\begin{proof}
From Theorem \ref{recursion1}, we know that the $f_n$'s satisfy the recursion formula \eqref{recursion-formula1}. Since $\{f_n(x)\}_{n=0}^{\infty}$ is a $q$-Appell polynomial sequence of type II, we have 
\[D_q^k f_n(x)=\dfrac{[n]_q!}{[n-k]_q!}f_{n-k}(q^kx),\quad 0\leq k\leq n.\]
It follows that 
\[f_{n-k}(x)=\dfrac{[n-k]_q!}{[n]_q!}\mathcal{E}_q^{-k}D_{q}^{k}f_{n}(x)= q^{\binom{k}{2}}\dfrac{[n-k]_q!}{[n]_q!}D_{q^{-1}}^k f_n(x),\quad 0\leq k\leq n.\]
Then \eqref{recursion-formula1} becomes 
\begin{eqnarray*}
f_n(x/q)&=&\dfrac{1}{[n]_q}\sum_{k=0}^{n}\qbinomial{n}{k}{q}\alpha_k q^{\binom{k}{2}}\dfrac{[n-k]_q!}{[n]_q!}D_{q^{-1}}^k f_n(x)+{q^{-1}}{x}f_{n-1}(x)\\
&=& \dfrac{1}{[n]_q}\sum_{k=0}^{n}\dfrac{q^{\binom{k}{2}}}{[k]_q!}\alpha_k D_{q^{-1}}^k f_n(x)+\dfrac{x}{q[n]_q} D_{q^{-1}}f_{n}(x),
\end{eqnarray*}
and the result follows
\end{proof}
\section*{Acknowledgements}
This work was supported by the Institute of Mathematics of the University of Kassel to whom I am very grateful.

\end{document}